\documentclass{article}%
\usepackage{amssymb}
\usepackage{amsfonts}
\usepackage{graphicx}
\usepackage{amsmath}%
\providecommand{\U}[1]{\protect\rule{.1in}{.1in}}
\newtheorem{theorem}{Theorem}

\newtheorem{corollary}[theorem]{Corollary}

\newtheorem{definition}[theorem]{Definition}

\newtheorem{lemma}[theorem]{Lemma}

\newtheorem{proposition}[theorem]{Proposition}
\newtheorem{remark}[theorem]{Remark}

\newenvironment{proof}[1][Proof]{\textbf{#1.} }{\ \rule{0.5em}{0.5em}}
\begin{document}
\[
\text{{\LARGE Reflexive\ Operator\ Algebras\ on\ Banach\ Spaces}}%
\]

\ \ \ \ \ \ \ \ \ \ \ \ \ \ \ \ \ \ \ \ \ \ \ \ \ \ \ \ \ \ \ \ \ \ \ \ \ \ \ \ \ \ \ \ 

\ \ \ \ \ \ by Florence Merlev\`{e}de$^{a}$, Costel
Peligrad\footnote{Supported in part by a Charles Phelps Taft Memorial Fund
grant.}$^{b}$ and Magda Peligrad\footnote{Supported in part by a Charles
Phelps Taft Memorial Fund grant, and the NSF grant DMS-1208237.}$^{c}$

\bigskip

\bigskip

$^{a}$ Florence Merlev\`{e}de: Universit\'{e} Paris Est, Laboratoire de
math\'{e}matiques, UMR 8050 CNRS, B\^{a}timent Copernic, 5 Boulevard
Descartes, 77435 Champs-Sur-Marne, France. E-mail: florence.merlevede@univ-mlv.fr

\bigskip

\bigskip

$^{b}$ Costel Peligrad: Department of Mathematical Sciences, University of
Cincinnati, PO Box 210025, Cincinnati, OH 45221-0025, USA. E-mail address: peligrc@ucmail.uc.edu

\bigskip

\bigskip

$^{c}$ Magda Peligrad: Department of Mathematical Sciences, University of
Cincinnati, PO Box 210025, Cincinnati, OH 45221-0025, USA. E-mail address: peligrm@ucmail.uc.edu

\bigskip\ \ \ \ \ \ \ \ \ \ \ \ \ \ \ \ \ \ \ \ \ \ \ \ \ \ \ \ \ \ \ \ \ \ \ \ \ \ \ \ \ \ \ \ \ \ \ \ \ \ \ \ \ \ \ \ \ \ \ \ \ \ \ \ \ \ \ \ \ \ \ \ 

\ \ \ \ \ \ \ \ \ \ \ \ \ \ \ \ \ \ \ \ \ \ \ \ \ \ \ \ \ \ \ \ \ \ \ \ \ \ \ \ \textbf{Abstract}%

\bigskip

In this paper we study the reflexivity of a unital strongly closed algebra of
operators with complemented invariant subspace lattice on a Banach space. We
prove that if such an algebra contains a complete Boolean algebra of
projections of finite uniform multiplicity and with the direct sum property,
then it is reflexive, i.e. it contains every operator that leaves invariant
every closed subspace in the invariant subspace lattice of the algebra. In
particular, such algebras coincide with their bicommutant.

\bigskip


\bigskip


\section{Introduction}

\bigskip

\quad Let $A\subset B(X)$ denote a strongly closed algebra of operators on the
Banach space X. Suppose that $A$ has the property that each of its invariant
subspaces has an invariant complement. If $A$ contains a complete Boolean
algebra of projections of finite uniform multiplicity and with the direct sum
property as defined below, we prove that $A$ is reflexive in the sense that it
contains all the operators which leave its closed invariant subspaces
invariant [Theorem 15]. In particular such an algebra is equal to its
bicommutant $A"$ [Corollary 22]$.$ The problem of whether a strongly closed
algebra of operators with complemented invariant subspace lattice is reflexive
started to be studied in the sixties. This problem is a generalization of the
invariant subsbspace problem in operator theory. In [1], Arveson introduced a
technique for studying the particular case of transitive algebras on Hilbert
spaces, namely the strongly closed algebras of operators on Hilbert spaces
that have no non-trivial closed invariant subspaces. He proved that every
transitive algebra that contains a maximal abelian von Neumann algebra
coincides with the full algebra $B(X)$ if $X$ is a complex Hilbert space. In
[5], Douglas and Pearcy extended the result of Arveson to the case of
transitive operator algebras containing an abelian von Neumann algebra of
finite multiplicity. In [9], Hoover extended the result of Douglas and Pearcy
to the case of reductive operator algebras on Hilbert spaces that contain
abelian von Neumann algebras of finite multiplicity. Hoover proved that every
reductive operator algebra (that is a strongly closed subalgebra for which
every closed invariant subspace is reducing) which contains an abelian von
Neumann algebra of finite multiplicity is self-adjoint. The transitive algebra
result of Douglas and Pearcy was generalized in [10] to the case of transitive
algebras on Banach spaces that contain a $n-$fold direct sum of a cyclic
complete Boolean algebra of projections. The case of operator algebras on
Banach spaces with complemented invariant subspace lattice was considered by
Rosenthal and Sourour in [12]. They proved that every strongly closed algebra
of operators with complemented invariant subspace lattice containing a
complete Boolean algebra of projections of uniform multiplicity one is reflexive.

In this paper we build upon the techniques introduced by Arveson, [1], and
developed\ by Douglas and Pearcy, [5], and Radjavi and Rosenthal, [11], for
invariant subspaces of operator algebras as well as Bade's multiplicity theory
of Boolean algebras of projections, [2, 3]. We also use the results of Foguel,
[8] and Tzafriri, [14], about the commutant of Boolean algebras of projections
of finite multiplicity.

\bigskip

\section{Notations and preliminary results}

\bigskip

\bigskip

\subsection{Invariant subspaces of operator algebras}

\bigskip

\quad Let $X$ be a complex Banach space and $B(X)$ the algebra of all bounded
linear operators on $X.$ We will denote by $X^{(n)}$ the direct sum of $n$
copies of $X$ and, if $S\subset B(X),$ we denote%
\[
S^{(n)}=\left\{  a\oplus a\oplus...\oplus a\in B(X^{(n)})\text{ ; }a\in
S\right\}  .
\]

If $S\subset B(Y),$ where $Y$ is a Banach space, we denote by$\ $%
\textit{Lat}$S$ the collection of all closed linear subspaces of $Y$ that are
invariant under every element of $S$. If $L$ is a collection of closed linear
subspaces of $Y$, we denote by \textit{alg}$L$ the (strongly closed) algebra
of operators on $Y$ that leave every element of $L$ invariant. An algebra
$A\subset B(X)$ is called \textit{reflexive }if \textit{algLat}$A=A.$

In what follows all the subalgebras $A\subset B(X)$ will be assumed to be
strongly closed and containing the identity operator $I\in B(X)$.

\begin{remark}
Let $A\subset B(X)$ be a strongly closed algebra with $I\in A$ and $b\in
B(X).$ If Lat$A^{(n)}\subset$Lat$b^{(n)}$ for every $n\in%
\mathbb{N}
$, then $b\in A.$
\end{remark}

\begin{proof}
Indeed, then for every finite set of elements $\left\{  x_{1},x_{2}%
,...,x_{n}\right\}  \subset X$ we have that $K=\overline{\left\{  ax_{1}\oplus
ax_{2}\oplus...\oplus ax_{n}\text{ ; }a\in A\right\}  }\in$\textit{Lat}%
$A^{(n)}$ and therefore $K\in$\textit{Lat}$b^{(n)}$. This means that $b\in A$,
since $A$ is strongly closed.
\end{proof}

\begin{proposition}
Let $A\subset B(X)$ be a strongly closed algebra with complemented invariant
subspace lattice and with $I\in A.$ Let $q\in B(X)$ be a projection$.$ Then
\newline(i) If $q\in A$, the algebra $qAq\subset B(qX)$ has complemented
invariant subspace lattice and algLat$(qAq)=q($algLat$A)q.$\newline(ii) If
$q\in A^{\prime},$ where $A^{\prime}$ denotes the commutant of $A,$ the strong
operator closure $\overline{qAq}^{so}\subset B(qX)$ is an algebra with
complemented invariant subspace lattice.
\end{proposition}

\begin{proof}
We prove first (i). Clearly, $qAq$ is a weakly closed subalgebra of $B(qX)$
whose unit is $q.$ Let $L\subset qX,$ $L\in$\textit{Lat}$\mathit{(}qAq).$ We
denote $\widetilde{L}=\overline{AqL}$, the closure being taken in $X.$ Then,
obviously, $\widetilde{L}\in$\textit{Lat}$A$ and, therefore $\widetilde{L}$
has a complement $\widetilde{L}^{c}$ in \textit{Lat}$A.$ Since $q\in A$, we
have $q\widetilde{L}\subset\widetilde{L}$, also $q\widetilde{L}^{c}%
\subset\widetilde{L}^{c}$ and $q\widetilde{L}^{c}\in$\textit{Lat}$(qAq)$.
Moreover, it is immediate that $q\widetilde{L}$ and $q\widetilde{L}^{c}$ are
closed linear subspaces of $qX$ such that%
\[
q\widetilde{L}\oplus q\widetilde{L}^{c}=qX.
\]
On the other hand we have $L\subset q\widetilde{L}=q\overline{AqL}%
\subset\overline{qAqL}\subset\overline{L}=L$. Hence\newline%
\[
q\widetilde{L}=L.
\]
It follows that $L$ is complemented in \textit{Lat}$(qAq)$ and so $qAq$ has
complemented invariant subspace lattice. Let now $b\in$ \textit{algLat}$A$ and
$L\in$\textit{Lat}$(qAq).$ By the above argument, there exists $\widetilde
{L}\in$\textit{Lat}$A$ such that $L=q\widetilde{L}.$ Hence $b\widetilde
{L}\subset\widetilde{L}$. Therefore, since $q\widetilde{L}=L\subset
\widetilde{L}$ it follows that $qbqL\subset L$, so $qbq\in$ \textit{algLat}%
$(qAq).$ Conversely, let $c\in$ \textit{algLat}$(qAq)$ and let $\widetilde
{c}\in B(X)$ be the extension of $c$ to $X$ that equals $0$ on $(I-q)X.$ Then,
it is straightforward to show that $\widetilde{c}\in$ \textit{algLat}$A$ and
$c=q\widetilde{c}q$ and so the proof is completed.\newline To establish (ii),
let $K\in$\textit{Lat}$(qAq)$. Since $q\in A^{\prime}$, it follows that $K\in
$\textit{Lat}$A$ and therefore $K$ has a complement $K^{c}\in$\textit{Lat}$A$.
Then, clearly $K^{c}\cap qX\in$\textit{Lat}$(qAq)$ and $K+K^{c}\cap qX=qX$.
\end{proof}

\bigskip

We will also need the following:

\begin{lemma}
Let $A\subset B(X)$ be an algebra with complemented invariant subspace lattice
and let $K\in$Lat$A$. If $p\in A^{\prime}$ is the projection on $K$ and
$t_{1},t_{2},...t_{n}\in(pAp)^{\prime}$, for some $n\in%
\mathbb{N}
,$ then the subspace
\[
\Gamma_{\left\{  t_{1},t_{2},...,t_{n};p\right\}  }=\left\{  x\oplus
t_{1}x\oplus t_{2}x\oplus...\oplus t_{n}x\text{ ; }x\in pX\right\}
\in\text{Lat}A^{(n+1)}%
\]
is complemented in Lat$A^{(n+1)}.$
\end{lemma}

\begin{proof}
Since $A$ has complemented invariant subspace lattice and $pX=K\in
$\textit{Lat}$A,$ it follows that the subspace $(1-p)X=(pX)^{c}=K^{c}%
$\ belongs to \textit{Lat}$A.$ It is then clear that $(pX)^{c}\oplus X^{(n)}$
is a complement of $\Gamma_{\left\{  t_{1},t_{2},...,t_{n};p\right\}  }$ in
\textit{Lat}$A^{(n+1)}.$
\end{proof}

\begin{remark}
Let $I\in A\subset B(X)$ be a strongly closed subalgebra with complemented
invariant subspace lattice. If $A$ is reflexive, then $A"=A$ where $A"$
denotes the bicommutant of $A.$
\end{remark}

\begin{proof}
If $a\in A"$, then, in particular, $a$ commutes with every projection on an
invariant subspace of $A.$ Therefore $a\in$ \textit{algLat}$A=A.$
\end{proof}

\bigskip The following concept is defined for instance in [11, \S \ 8.2].

\begin{definition}
Let $A\subset B(X)$ be a subalgebra. A linear operator $T$ defined on a not
necessarily closed linear subspace $P\subset X$ is called a graph
transformation for $A$ if there exist finitely many linear operators
$T_{1},T_{2},...,T_{l},$ all defined on $P,$ such that
\[
\left\{  x\oplus Tx\oplus T_{1}x\oplus T_{2}x\oplus...\oplus T_{l}x\text{ ;
}x\in P\right\}  \in\text{Lat}A^{(l+2)}.
\]

\end{definition}

\begin{remark}
Let $K\in$Lat$A^{(n)},n\in%
\mathbb{N}
.$ Denote by
\[
K_{0}=\left\{  \mathbf{x}\in X^{(n-1)}\text{ ; }0\oplus\mathbf{x}\in
K\right\}  \in\text{Lat}A^{(n-1)}.
\]
Then, if $K_{0}$ is complemented in Lat$A^{(n-1)}$ with complement $K_{0}^{c}$
it follows that there exist graph transformations for $A:$ $T_{1}%
,T_{2},...,T_{n-1},$ defined on a linear subspace, $P\subset X,$ such that%
\[
(X\oplus K_{0}^{c})\cap K=\left\{  x\oplus T_{1}x\oplus T_{2}x\oplus...\oplus
T_{n-1}x\text{ ; }x\in P\right\}  .
\]

\end{remark}

\begin{proof}
Straightforward.
\end{proof}

\subsection{Boolean algebras of projections in Banach spaces and spectral
operators}

\bigskip

\quad Let $\mathcal{B}$ be complete Boolean algebra of projections in a
(complex) Banach space $X$ (as defined for instance in [2] or [6, Chapter
XVII]). It is known, [13], that there exists an extremally disconnected
compact Hausdorff topological space $\Omega$ (that is a totally disconnected
compact Hausdorff space in which the closure of every open set is also open),
such that $\mathcal{B}$ is equivalent as a Boolean algebra with the Boolean
algebra of open and closed subsets of $\Omega$. We will denote by $\Sigma$ the
collection of Borel sets of $\Omega.$ Such a compact Hausdorff space is called
a Stonean space.

\bigskip

The following remark collects some results about the complete Boolean algebras
of projections in Banach spaces that will be used in this paper.

\begin{remark}
(i) If $\mathcal{B}$ is a complete Boolean algebra of projections, then there
is a regular countably additive spectral measure $E$ in $X$ defined on the
family of Borel sets in $\Omega$ such that the mapping%
\[
S(f)=\int\limits_{\Omega}f(w)E(dw)
\]
is a continuous isomorphism of the algebra $C(\Omega)$ of continuous functions
on $\Omega$ \ onto the the uniformly closed algebra of operators, $B,$
generated by $\mathcal{B}$.\newline(ii) The algebra $B$ coincides with the
weakly closed algebra generated by $\mathcal{B}$ and consists of spectral
operators of scalar type.\newline(iii) The range of $E$ is precisely the
Boolean algebra $\mathcal{B}.$\newline(iv) $\mathcal{B}$ is norm bounded.
\end{remark}

\begin{proof}
(i) and (iii) follow from [ 6, Lemma XVII 3.9.], (ii) is [ 6, Corollary XVII
3.17.] and (iv) follows from [ [2], Theorem 2.2.].
\end{proof}

Let $\mathcal{B}\subset B(X)$ be a complete Boolean algebra of projections. We
say that $\mathcal{B}$ has uniform multiplicity $k,$ $k\in%
\mathbb{N}
$ if there are $x_{1},x_{2},...,x_{k}\in X$ such that $\overline{lin}\left\{
ex_{i}\text{ ; }e\in\mathcal{B},1\leq i\leq k\right\}  =X$ and no subset of
$X$ of cardinality less than $k$ has this property [3, Definition 3.2]. For
each $i,1\leq i\leq k,$ denote $\mathfrak{M}(x_{i})=\overline{lin}\left\{
ex_{i}\text{ ; }e\in\mathcal{B}\right\}  .$ Here, $\overline{lin}\left\{
ex_{i}\text{ ; }e\in\mathcal{B}\right\}  $ denotes the closed linear subspace
of $X$ spanned by $\left\{  ex_{i}\text{ ; }e\in\mathcal{B}\right\}  .$

\bigskip

The next remark collects some known results from [3], (see also [6]):

\bigskip

\begin{remark}
Let $\mathcal{B}$ be a complete Boolean algebra of finite uniform multiplicity
$n,n\in%
\mathbb{N}
$ and let $\left\{  x_{1},\cdots,x_{n}\right\}  $ be a set of vectors such
that
\[
\overline{lin}\left\{  ex_{i}\text{ ; }e\in\mathcal{B},1\leq i\leq n\right\}
=X
\]
Then \newline(i) There are $x_{i}^{\ast}\in X^{\ast},i=1,2,...,n,$ where
$X^{\ast}$ is the dual Banach space of $X,$ such that each of the measures
$\mu_{i}(\delta)=x_{i}^{\ast}E(\delta)x_{i},i\in\left\{  1,2,...,n\right\}  ,$
$\delta\in\Sigma$ vanishes on sets of first category of $\Omega$ and $\mu
_{i}(\sigma)\neq0$ if $\sigma$ has nonempty interior. The measures $\mu_{i}$
are equivalent and $x_{i}^{\ast}(\mathfrak{M}(x_{j}))=\left\{  0\right\}  $
for $i\neq j.$\ \newline(ii) There exists a continuous injective linear map
$V$ of $X$ onto a dense linear subspace $L\subset\sum\limits_{i=1}^{n}%
L^{1}(\Omega,\Sigma,\mu_{i})$ such that if $V(x)=\mathbf{f}=\sum f_{i}$,
then\newline(a) $x_{i}^{\ast}E(\delta)x=\int_{\delta}f_{i}(\omega)\mu
_{i}(d\omega),$ $\delta\in\Sigma.$ In particular $V(x_{i})=0\oplus0\oplus
\chi_{\Omega}\oplus...\oplus0,$ where $\chi_{\Omega}=1$ is on the $i$-th place
in the direct sum.\newline(b) $x=\lim_{m\rightarrow\infty}\sum_{i=1}%
^{n}S(f_{i}\chi_{\delta_{m}})x_{i}$ where $\chi_{\delta_{m}}$ is the
characteristic function of $\delta_{m}=\left\{  \omega\text{ ; }\left\vert
f(\omega)\right\vert \leq m\right\}  .$\newline(iii) The linear space $L$
endowed with the norm $\left\Vert \mathbf{f}\right\Vert _{0}=\max
\limits_{1\leq i\leq n}\left\Vert f_{i}\right\Vert _{1}+\left\Vert
V^{-1}(\mathbf{f})\right\Vert $, is a Banach space and $V$ is a Banach space
isomorphism between $X$ and $(L,\left\Vert \cdot\right\Vert _{0}).$
\end{remark}

\begin{proof}
The points (i) and (ii) follow from [3, Lemma 5.1. and Theorem 5.2.], (see
also [6, Theorem XVIII 3.19.] ). The proof of (iii) is immediate.
\end{proof}

A function $f$ is called $E-$essentially bounded if%
\[
\inf_{E(\delta)=1}\sup_{\omega\in\delta}\left\vert f(\omega)\right\vert
\]
is finite [6, Definition 7].

Denote by $EB(\Omega,\Sigma)$ the set of all $E-$essentially bounded $\Sigma
-$measurable functions.

\begin{lemma}
With the notations in Remark 8, if $\varphi\in EB(\Omega,\Sigma)$, then the
operator $M_{\varphi}(\mathbf{f})=\varphi\mathbf{f}$\textbf{ }is a well
defined, bounded operator on $(L,\left\Vert \cdot\right\Vert _{0})$ and
$M_{\varphi}=VS(\varphi)V^{-1}.$ Here $\varphi\mathbf{f=}\varphi f_{1}%
\oplus\varphi f_{2}\oplus...\oplus\varphi f_{n}$. Thus
\[
VBV^{-1}=\left\{  M_{\varphi}\text{ ; }\varphi\in EB(\Omega,\Sigma)\right\}
.
\]

\end{lemma}

\begin{proof}
Let $\mathbf{f}\in L$ and $x=S(\varphi)V^{-1}(\mathbf{f}).$ Then, according to
point (a) in Remark 8(ii), if $\mathbf{g}=V(x)$, we have $x_{i}^{\ast}%
S(\chi_{\delta})x=\int_{\delta}g_{i}(w)\mu_{i}(dw)$ for every Borel set
$\delta\in\Sigma.$ On the other hand, $x_{i}^{\ast}S(\chi_{\delta}%
)x=x_{i}^{\ast}S(\chi_{\delta})S(\varphi)V^{-1}(\mathbf{f})=x_{i}^{\ast}%
S(\chi_{\delta}\varphi)V^{-1}(\mathbf{f})=\int_{\delta}\varphi(w)f_{i}%
(w)\mu_{i}(dw).$ Hence $g_{i}=\varphi f_{i}$ $\mu_{i}-$a.e., so $\mathbf{g}%
=\varphi\mathbf{f}$ a.e. and the proof is completed.
\end{proof}

\bigskip

In [4] it is presented an example of a Boolean algebra of projrctions,
$\mathcal{B}$, such that every nonzero projection $e\in\mathcal{B}$ has
multiplicity 2. However, for no choice of $x_{1},x_{2}\in X$ or $e\in
\mathcal{B}$, $e\neq0$ is $eX$ the algebraic sum of $\mathfrak{M}(ex_{1})$ and
$\mathfrak{M}(ex_{2}).$ In the rest of this paper we will consider only
Boolean algebras of finite uniform multiplicity with the direct sum property:

\begin{definition}
We say that the complete Boolean algebra $\mathcal{B}$ of uniform multiplicity
$k$ has the direct sum property if $X$ is the algebraic (and therefore,
Banach) direct sum of $\ \mathfrak{M}(x_{i}),1\leq i\leq k.$
\end{definition}

\bigskip

A particular case of a Boolean algebra of uniform multiplicity $k$ with the
direct sum property is the $k-$fold direct sum of $k$ copies of a cyclic
Boolean algebra of projections. Other examples are presented in [7].

\bigskip

\begin{lemma}
Suppose that $\mathcal{B}$ is a complete Boolean algebra of projections of
uniform multiplicity $k$ with the direct sum property. Then, for every
$\epsilon>0$ there exist $e\in\mathcal{B}$ , $e=E(\rho),$ $\rho\in\Sigma$ with
$\mu_{l}(\rho^{c})<\epsilon$ for every $1\leq l\leq k$ (where $\rho^{c}$ is
the complement of $\rho$) such that for every $\left\{  \varphi_{ij};1\leq
i,j\leq k\right\}  \subset EB(\Omega,\Sigma)$, the matrix $\left[
\varphi_{ij}\chi_{\rho}\right]  $ is a bounded linear operator on
$(L,\left\Vert \cdot\right\Vert _{0})$ and $\left[  \varphi_{ij}\chi_{\rho
}\right]  \ $\ belongs to the commutant $\mathcal{B}^{\prime}$ of
$\mathcal{B}$.
\end{lemma}

\begin{proof}
Since the measures $\mu_{l},$ $1\leq l\leq k$ are equivalent, let
$h_{ml}=\frac{d\mu_{m}}{d\mu_{l}},$ $1\leq m,l\leq k$ be the corresponding
Radon Nikodym derivative. Let $\epsilon>0$ be arbitrary. Fix $1\leq m,l\leq
k$. Then, since $\cup_{n=1}^{\infty}\left\{  1/n\leq h_{ml}\leq n\right\}
=\Omega$, there is a $n\in%
\mathbb{N}
$ such that $\mu_{l}(\left\{  1/n\leq h_{ml}\leq n\right\}  ^{c}%
)<\epsilon/k^{2}$. Therefore there is a $n\in%
\mathbb{N}
$ such that $\mu_{l}(\left\{  1/n\leq h_{ml}\leq n\right\}  ^{c})<\epsilon$
for every $1\leq m,l\leq k.$ Let $\rho=\left\{  1/n\leq h_{ml}\leq n\right\}
\in\Sigma$. It is easy to see that for every Borel subset $\sigma\subset\rho$
we have $\mu_{i}(\sigma)/n\leq\mu_{j}(\sigma)\leq n\mu_{i}(\sigma)$ for all
$1\leq i,j\leq k$. Hence all the spaces $M_{\chi_{\rho}}L^{1}(\mu_{i}%
)=\chi_{\rho}L^{1}(\mu_{i}),$ $1\leq i\leq k$, are equal as sets and mutually
isomorphic as Banach spaces. Then, clearly,%
\[
\chi_{\rho}L=\chi_{\rho}L^{1}(\mu_{1})\oplus\chi_{\rho}L^{1}(\mu_{2}%
)\oplus...\oplus\chi_{\rho}L^{1}(\mu_{k})
\]
Since $\mathcal{B}$ has the direct sum property, we also have%
\[
E(\rho)X=E(\rho)\mathfrak{M}(x_{1})\oplus\cdots\oplus E(\rho)\mathfrak{M}%
(x_{k})
\]
and the lemma follows.
\end{proof}

\bigskip

For the definition and basic facts about spectral operators on Banach spaces
we refer to [6, Part III]. In what follows we will use the following result of
Foguel [8] and Tzafriri [14]:

\begin{remark}
Let $T\in B(X)$ and let $\mathcal{B}$ be a complete Boolean algebra of
projections in $X,$ of uniform multiplicity $k,k\in%
\mathbb{N}
$. If $T$ commutes with the weakly closed algebra $B,$ generated by
$\mathcal{B},$ then there exists an increasing sequence of projections
$\left\{  e_{m}=E(\chi_{\delta_{m}})\text{ ; }m\in%
\mathbb{N}
\right\}  \subset\mathcal{B}$ such that $\left\{  e_{m}\right\}  $ converges
strongly to the identity $I\in B(X)$ and $Te_{m}$ is a spectral operator of
finite type for every $m$. Moreover, if $T\in B^{\prime}$ is a spectral
operator then $T$ is the sum of a spectral operator $R$ of scalar type in
$B^{\prime}$ and a nilpotent operator $Q$ of order $k,$ $Q\in B^{\prime}$.
\end{remark}

\begin{proof}
This follows from [14, Theorem 2 ] and [ 8, Lemma 2.1. and Theorem 2.3].
\end{proof}

\bigskip

Next we will study the dense linear subspaces of $X$ that are invariant under
every element of $B,$ where $B$ is the weakly closed algebra generated by
$\mathcal{B}$, the complete Boolean algebra of projections of uniform
multiplicity $k$ with the direct sum property. The following lemma is an
extension to the case of Banach spaces and an improvement on [5, Lemma 3.3].
Using Remark 8 and Lemma 9, we will identify\ $X$ with $L$ and $B$ with
$\left\{  VS(\varphi)V^{-1}\text{ ; }S(\varphi)\in B\right\}  .$

\bigskip

\begin{lemma}
Let $k\in%
\mathbb{N}
$ and $B$ the weakly closed algebra generated by the Boolean algebra of
projections of uniform multiplicity $k,$ $\mathcal{B}\subset B(X)$ and with
the direct sum property. With the above notations, suppose that
$\mathcal{D\subset}X$ is a dense linear subspace which is invariant under all
operators in $B.$ Then, for every $\epsilon>0,$ there exists an open and
closed set $\lambda_{\epsilon}\subset\Omega$ such that:\newline(i) $\mu
_{i}(\lambda_{\epsilon}^{c})<\epsilon,i=1,2,...,k,$ where $\lambda_{\epsilon
}^{c}$ is the complement of $\lambda_{\epsilon}$ in $\Omega$ and\newline(ii)
$\chi_{\lambda_{\epsilon}}\mathbf{e}_{j}\in\mathcal{D}$ for all $j\in\left\{
1,2,...,k\right\}  $ where $\left\{  \mathbf{e}_{j}\text{ ; }%
j=1,2,...,k\right\}  $ is the standard basis of $%
\mathbb{C}
^{(k)}.$
\end{lemma}

\begin{proof}
If $\mathbf{z=(}z^{1},z^{2},...,z^{k})\in%
\mathbb{C}
^{(k)},$ consider the norm\newline%
\[
\left\Vert \mathbf{z}\right\Vert =\max\left\{  \left\vert z^{p}\right\vert
;1\leq p\leq k\right\}  .
\]
It is easy to see that there exists $\alpha>0$ such that if the set $\left\{
h_{1},h_{2},...,h_{k}\right\}  \subset%
\mathbb{C}
^{(k)}$ satisfies $\left\Vert h_{i}-e_{i}\right\Vert <\alpha,$ $i=1,2,...,k,$
then the set $\left\{  h_{1},h_{2},...,h_{k}\right\}  $ is linearly
independent. Let now $\epsilon>0$ be arbitrary. We can choose $\alpha
<\epsilon^{2}/2.$ Let $\rho\in\Sigma$, $\mu_{l}(\rho)<\epsilon/2,1\leq l\leq
k,$ be as in Lemma 11$.$ Since $\Omega$ is extremally disconnected, we can
assume that $\rho$ is an open and closed set. For every $j,$ $1\leq j\leq k$
let $\mathbf{g}_{j}(w)=\mathbf{e}_{j},$ if $w\in\rho$ and $\mathbf{g}%
_{j}(w)=0$ if $\omega\in\rho^{c}.$ Since by the point (a) of Remark 8(ii) we
have that $\mathbf{g}_{j}\in\chi_{\rho}L$ for every $j,$ $1\leq j\leq k$ and
$\chi_{\rho}\mathcal{D}$ is dense in $\chi_{\rho}L$, it follows that there
exists a set of elements $\left\{  \mathbf{l}_{i}\text{\ ; }1\leq i\leq
k\right\}  \subset\chi_{\rho}\mathcal{D},\mathbf{l}_{i}=l_{i}^{1}\oplus
l_{i}^{2}\oplus...\oplus l_{i}^{k}$ such that%
\[
\left\Vert \mathbf{l}_{i}-\mathbf{g}_{i}\right\Vert _{0}=\max_{1\leq p\leq
k}\left\{  \left\Vert l_{i}^{p}-g_{i}^{p}\right\Vert _{0}=\left\Vert l_{i}%
^{p}-g_{i}^{p}\right\Vert _{1}+\left\Vert T^{-1}(l_{i}^{p}-g_{i}%
^{p}\right\Vert \right\}  <\alpha<\epsilon^{2}.
\]
\newline Let $\delta_{\epsilon}=\cap_{i=1}^{k}\left\{  \omega\in\rho\text{ ;
}\left\vert l_{i}^{p}(w)-g_{i}^{p}(w)\right\vert \geq\epsilon\text{ and }1\leq
p\leq k\right\}  .$ Then, we have:%
\[
\epsilon^{2}/2>\alpha>\max\left\{  \left\Vert l_{i}^{p}-g_{i}^{p}\right\Vert
_{1};\text{ }1\leq i,p\leq k\right\}  \geq\epsilon\mu_{m}(\delta_{\epsilon
}),\text{ for }1\leq m\leq k.
\]
Hence $\mu_{m}(\delta_{\epsilon})<\epsilon/2$ for $m=1,2,...,k.$ Assuming that
$\epsilon<2,$ it follows that $\mu_{m}(\delta_{\epsilon}^{c})\neq0$ and since
$\Omega$ is a Stonean space, and $\mu_{m}$ a normal measure, $\mu_{m}%
(\delta_{\epsilon}^{c})=\mu_{m}((\delta_{\epsilon}^{c})^{\circ})$ where
$(\delta_{\epsilon}^{c})^{\circ}$ is the interior of $\delta_{\epsilon}^{c}.$
The same argument as the preceding one shows that there exists an open and
closed subset $\sigma_{\epsilon}\subset(\delta_{\epsilon}^{c})^{\circ}$ with
$\mu_{m}(\sigma_{\epsilon}^{c})<\epsilon/2$. Let $\lambda_{\epsilon}=\rho
\cap\sigma_{\epsilon}.$ Then, $\mu_{m}(\lambda_{\epsilon})<\epsilon$ for all
$1\leq m\leq k.$ It follows that all the components of the vectors
$l_{i}^{\epsilon}=l_{i}\chi_{\lambda_{\epsilon}}\in L$ are in $EB(\Omega
,\Sigma).$ Let $M$ be the matrix whose $i$-th column is $l_{i}^{\epsilon}.$
Then, using Lemma 11, it follows that $M$ is a bounded linear operator that
commutes with every element in $B$, so $M\in B^{^{\prime}}.$ The choice of
$\alpha$ implies that $M(w)$ is non singular for every $\omega\in
\lambda_{\epsilon}.$ Consider the matrix $N$ defined as follows%
\[
N(w)={\Huge \{}%
\begin{array}
[c]{c}%
M(w)^{-1}\text{ \ \ \ if }w\in\lambda_{\epsilon}\\
\text{ \ \ \ }0\text{ \ \ \ \ \ \ \ \ \ if }w\in\lambda_{\epsilon}^{c}%
\end{array}
\]
\newline By restricting $N$ to an open and closed subset of $\lambda
_{\epsilon},$ if necessary, we can apply again Lemma 11 and get $N\in
B^{\prime}$. It follows that the columns of the product $MN$ are linear
combinations of vectors in $\mathcal{D}$ with coefficients in $B.$ Since
$\mathcal{D}$ is invariant under $B$ we have that these columns belong to
$\mathcal{D}$. Since $M(w)N(w)=I$ for $w\in\lambda_{\epsilon}$ the proof is completed.
\end{proof}

\bigskip

We will use next the following results about spectral operators and their
resolutions of the identity from [6].

\begin{remark}
If the operator $M$ commutes with the spectral operator $T$, then $M$ commutes
with every resolution of the identity of $T.$
\end{remark}

\begin{proof}
This is [ [6], Corollary XV. 3.7.].
\end{proof}

\section{Algebras with complemented invariant subspace lattice}

\bigskip

\quad In this section we will prove our main result:

\begin{theorem}
Let $B$ be the strongly closed subalgebra of $B(X)$ generated by a complete
Boolean algebra of projections $\mathcal{B}\subset B(X)$ of finite uniform
multiplicity, $k,$ with the direct sum property$.$ If $A\subset B(X)$ is a
strongly closed algebra with complemented invariant subspace lattice that
contains $B,$ then $A$ is reflexive.
\end{theorem}

The proof of this theorem will be given after a series of auxiliary results.
In the rest of this section $\mathcal{B}\mathbb{\ }$\ will denote a complete
Boolean of projections in $X$ of finite uniform multiplicity $k,$ with the
direct sum property and $B$ the weakly closed operator algebra generated by
$\mathcal{B}$. We will identify $X$ with $(L,\left\Vert \cdot\right\Vert
_{0})$ as in Remark 8.

\begin{proposition}
Let $B$ as in Theorem 15 and let $T$ be a densely defined closed operator on
$X$ which commutes with $B$. There exists an increasing sequence of
projections $\left\{  q_{p}\right\}  _{p=1}^{\infty}\subset\mathcal{B}$ that
converges strongly to $I$ such that $Tq_{p}$ is a spectral operator of finite
type for every $p\in%
\mathbb{N}
$.
\end{proposition}

\begin{proof}
Let $\mathcal{D}\subset X$ be the (dense) domain of $T.$ Since $T$ commutes
with $B$ it follows that $\mathcal{D}$ is invariant under $B.$ By Lemma 13 it
follows that for every $p\in%
\mathbb{N}
$ there is an open and closed subset $\sigma_{p}\subset\Omega$ such that
$\chi_{\sigma_{p}}\oplus\chi_{\sigma_{p}}\oplus...\oplus\chi_{\sigma_{p}}%
\in\mathcal{D}$ and $\mu_{l}(\sigma_{p}^{c})<1/2p$ for every $1\leq l\leq k$.
Denote $r_{p}=S(\chi_{\sigma_{p}})\in\mathcal{B}$. Obviously, we can take
$r_{p}\leq r_{p+1}$ (in the sense that $r_{p}X\subset r_{p+1}X)$ for every
$p\in%
\mathbb{N}
$. Therefore $Tr_{p}$ $(p\in%
\mathbb{N}
)$ is a bounded operator and $r_{p}\nearrow I$. On the other hand, by Remark
12, since $Tr_{p}\in B^{^{\prime}}$, for every $p\in%
\mathbb{N}
,$ there exists a Borel set $\delta_{p}\in\Sigma$ such that, for all $1\leq
l\leq p,$ we have $\mu_{l}(\delta_{p}^{c})<1/2p.$ Furthermore, if
$q_{p}=S(\chi_{\delta p\cap\sigma_{p}})$, then $Tq_{p}$ is a spectral operator
of finite type. Clearly $\left\{  q_{p}\right\}  $ is an increasing sequence
of projections in $\mathcal{B}$ that converges strongly to $I$ and the proof
is completed.
\end{proof}

\begin{proposition}
Assume that $B$ is as in the statement of Theorem 15. Let $T$ be a densely
defined graph transformation for $B\subset B(X)$. Then there exists an
increasing sequence of projections $\left\{  q_{p}\right\}  _{p=1}^{\infty
}\subset\mathcal{B}$ that converges strongly to $I$ such that $Tq_{p}$ is a
spectral operator of finite type for every $p\in%
\mathbb{N}
$. In\ particular every such transformation is closable\ and its closure
commutes with\ $B$.
\end{proposition}

\begin{proof}
Let $T$ be a densely defined graph transformation for $B$ with domain
$\mathcal{D}_{T}$. Since $T$ is a graph transformation for $B$, there exists
$l\in%
\mathbb{N}
$ and operators $T_{1},T_{2},...,T_{l-2}$ such that the subspace
\[
Z=\left\{  x\oplus Tx\oplus T_{1}x\oplus T_{2}x\oplus...\oplus T_{l-2}x\text{
};\text{ }x\in\mathcal{D}_{T}\right\}
\]
belongs to \textit{Lat}$B^{(l)}$. Denote by $\Delta_{l-1}=\left\{  x\oplus
x\oplus...\oplus\text{ ; }x\in X\right\}  \subset X^{(l-1)}$. Then it can be
easily seen that the subspace
\[
\Delta_{l-1}^{c}=\left\{  x_{1}\oplus x_{2}\oplus...\oplus x_{l-1}\text{ ;
}x_{i}\in X\text{ with }\sum_{i=1}^{l-1}x_{i}=0\right\}
\]
is a Banach subspace complement of $\Delta_{l-1}$ which is invariant under
every element of $B^{(l-1)}$. The operator $\widetilde{T}$ defined as follows%
\[
\widetilde{T}(x\oplus x\oplus...\oplus x)=Tx\oplus T_{1}x\oplus...\oplus
T_{l-2}x\text{ \ \ \ \ \ \ \ if }x\in\mathcal{D}_{T}\text{ }%
\]
and\newline%
\[
\widetilde{T}(x_{1}\oplus x_{2}\oplus...\oplus x_{l-1})=0\text{ \ \ \ if
}x_{1}\oplus x_{2}\oplus...\oplus x_{l-1}\in\Delta_{l-1}^{c}%
\]
is a closed, densely defined operator which commutes with $B^{(l-1)}$. An
application of Proposition 16 with $k\ $replaced by $k(l-1)$ completes the proof.
\end{proof}

\begin{remark}
Let $A\subset B(X)$ be a strongly closed algebra with complemented invariant
subspace lattice and $I\in A$. Then, if $Q\in A^{\prime}$ is such that
$Q^{2}=0$ it follows that $Q\in($algLat$A)^{\prime}$.
\end{remark}

\begin{proof}
The proof of [7, Lemma 3] for the particular case of Hilbert spaces can be
extended to the case of Banach spaces. Indeed, let $Q\in A^{\prime}$ be such
that $Q^{2}=0.$ Then, if $Y=\ker Q$ is the null space of $Q,$ $Y$ is in
\textit{Lat}$A$ and since $A$ has a complemented invariant subspace lattice,
$Y$ has a complement, $Y^{c}$ in \textit{Lat}$A.$ Therefore $Q$ can be written
as a matrix%

\[
Q=\left[
\begin{array}
[c]{cc}%
0\text{ } & c\\
0 & 0\text{ }%
\end{array}
\right]  .
\]
\newline and every $a\in A$ can be written as the matrix\newline%
\[
a=\left[
\begin{array}
[c]{cc}%
a_{1} & 0\\
0 & a_{2}%
\end{array}
\right]  .
\]
\newline Moreover, every $b\in$ \textit{algLat}$A$, can be written as a
matrix\newline%
\[
b=\left[
\begin{array}
[c]{cc}%
b_{1} & 0\\
0 & b_{2}%
\end{array}
\right]
\]

Since $aQ=Qa$ it follows that $ca_{2}=a_{1}c.$ Hence the subspace $\left\{
cx\oplus x\text{ ; }x\in Y^{c}\right\}  $ belongs to \textit{Lat}$A$ and is
therefore invariant for \textit{algLat}$A.$ It follows that $cb_{2}=b_{1}c,$
so $Qb=bQ$.
\end{proof}

\bigskip

The part (i) of the next result is a generalization of Remark 18.

\begin{proposition}
Let $A\subset B(X)$ be an algebra with complemented invariant subspace
lattice. Then \newline(i) If $Q\in A^{\prime}$ is a nilpotent operator, then
$Q\in($algLat$A)^{\prime}.$\newline(ii) If $T=R+Q$ is a spectral operator of
finite type (where $R$ is spectral of scalar type and $Q$ is nilpotent) and
$T\in A^{\prime}$, then $R\in($algLat$A)^{\prime}$ and $N\in($%
algLat$A)^{\prime}.$
\end{proposition}

\begin{proof}
We will prove the point (i) of this proposition by induction. By Remark
18(ii), if $Q\in A^{\prime}$ and $Q^{2}=0$, then $Q\in($\textit{algLat}%
$A)^{\prime}$. Suppose that for every operator $Q\in A^{\prime}$ with
$Q^{n}=0$ it follows that $Q\in($\textit{algLat}$A)^{\prime}$ and let $Q\in
A^{\prime}$ with $Q^{n+1}=0$. Let $p_{0}$ denotes a projection on $\ker Q$
such that $p_{0}\in A^{\prime}$. Since $Qp_{0}=0$ it follows that \newline%
\[
(1-p_{0})Q=(1-p_{0})Q(1-p_{0})
\]
\newline and therefore\newline%
\[
(1-p_{0})Q^{k}=((1-p_{0})Q(1-p_{0}))^{k},k\in%
\mathbb{N}
.
\]
\newline Since $Q^{n+1}=0$ we have $Q^{n}(X)\subset\ker Q$ and
therefore\newline%
\[
0=(1-p_{0})Q^{n}=((1-p_{0})Q(1-p_{0}))^{n}.
\]
\newline By hypothesis, $(1-p_{0})Q=(1-p_{0})Q(1-p_{0})\in($\textit{algLat}%
$A)^{\prime}.$ On the other hand, since $Q\in A^{\prime}$ and $p_{0}\in
A^{\prime}$ we have $p_{0}Q\in A^{\prime}.$ Since obviously $(p_{0}Q)^{2}=0$,
by Remark 18 (ii), it follows that $p_{0}Q\in($\textit{algLat}$A)^{\prime}$.
Therefore\newline%
\[
Q=p_{0}Q+(1-p_{0})Q\in(\text{\textit{algLat}}A)^{\prime}%
\]
and the proof of (i) is completed.\newline We turn now to prove the point
(ii). By Remark 14, every resolution of the identity of $T$, $\mathbf{E}%
(\delta)$, where $\delta$ is a Borel subset of the spectrum of $T,$
$\delta\subset sp(T),$ is in $A^{\prime}.$ Therefore, since $A$ has
complemented invariant subspace lattice, it follows that $\mathbf{E}%
(\delta)\in($\textit{algLat}$A)^{\prime}$ for every Borel set $\delta\subset
sp(T).$ Hence $R=\int\lambda\mathbf{E}(d\lambda)\in($\textit{algLat}%
$A)^{\prime}.$ Since $T\in A^{\prime}$ and $R\in A^{\prime}$ it follows that
$Q\in A^{\prime}.$ By part (i) it follows that $Q\in($\textit{algLat}%
$A)^{\prime}$.
\end{proof}

\begin{lemma}
Let $A$ be a strongly closed algebra with complemented invariant subspace
lattice that contains a complete Boolean algebra of projections of finite
uniform multiplicity $k$ with the direct sum property. Then, if $K\in
$Lat$A^{(n)}$ for some $n\in%
\mathbb{N}
$, then, there exists an increasing sequence of projections $\left\{
p_{m}\right\}  \subset\mathcal{B},$ $p_{m}\nearrow I$ such that $p_{m}^{(n)}K$
is complemented in Lat$(p_{m}Ap_{m})^{(n)}$ for every $m\in%
\mathbb{N}
$.
\end{lemma}

\begin{proof}
We will prove the lemma by induction on $n.$ For $n=1$ the statement is
obvious with $p_{m}=I$ for every $m.$ Let $K\in$\textit{Lat}$A^{(n)}.$ Denote
$K_{0}=\left\{  \mathbf{x}\in X^{(n-1)}\text{ ; }0\oplus\mathbf{x}\in
K\right\}  .$ Then, obviously, $K_{0}\in$\textit{Lat}$A^{(n-1)}$ and therefore
there exists an increasing sequence of projections $\left\{  r_{m}\right\}
\subset\mathcal{B}$, $r_{m}\nearrow I$ such that $r_{m}^{(n-1)}K_{0}$ is
complemented in \textit{Lat}$(r_{m}Ar_{m})^{(n-1)}$. Let $(r_{m}^{(n-1)}%
K_{0})^{c}$ be the complement of $r_{m}^{(n-1)}K_{0}$ in \textit{Lat}%
$(r_{m}Ar_{m})^{(n-1)}$. Then, $(r_{m}X\oplus(r_{m}^{(n-1)}K_{0})^{c})\cap
K\in$\textit{Lat}$A^{(n)}$ and $r_{m}^{(n)}K=(0\oplus r_{m}^{(n-1)}%
K_{0})+(r_{m}X\oplus(r_{m}^{(n-1)}K_{0})^{c})\cap r_{m}^{(n)}K.$ Since
$(r_{m}X\oplus(r_{m}^{(n-1)}K_{0})^{c})\cap r_{m}^{(n)}K$ is the complement of
$0\oplus r_{m}^{(n-1)}K_{0}$ in $r_{m}^{(n)}K,$ there exist graph
transformations $T_{1},T_{2},...T_{n-1}$ such that
\[
(r_{m}X\oplus(r_{m}^{(n-1)}K_{0})^{c})\cap r_{m}^{(n)}K=\left\{  x\oplus
T_{1}x\oplus T_{2}x\oplus...\oplus T_{n-1}x\text{ };\text{ }x\in P\right\}  ,
\]
where $P$ is a linear subspace of $r_{m}X$ invariant under every element of
$r_{m}Ar_{m}$. The closure of $P$ in $r_{m}X$, $\overline{P},$ belongs to
\textit{Lat}$(r_{m}Ar_{m})$ and hence has a complement $\overline{P}^{c}$ in
\textit{Lat}$(r_{m}Ar_{m})$. For $1\leq i\leq n-1,$ consider the following
densely defined, graph transformation on $r_{m}X$:\newline%
\begin{align*}
\widetilde{T}_{i}x  &  =T_{i}x\text{ \ if }x\in P\text{ and}\\
\widetilde{T}_{i}x  &  =0\text{ \ \ \ \ if }x\in\overline{P}^{c}.
\end{align*}
\newline Then $\widetilde{T}_{i}$ commutes with $A.$ By Proposition 17, there
exists an increasing sequence of projections $\left\{  q_{p}\right\}
\subset\mathcal{B}$, $q_{p}\nearrow I$ such that $\widetilde{T}_{i}q_{p}$ are
bounded spectral operators of finite type. From Lemma 3 it follows that the
subspace\newline%
\[
\Gamma_{\left\{  \widetilde{T}_{1}q_{p},\widetilde{T}_{2}q_{p},\cdots
,\widetilde{T}_{n-1}q_{p}\right\}  }=\left\{  q_{p}x\oplus\widetilde{T}%
_{1}q_{p}x\oplus\widetilde{T}_{2}q_{p}x\oplus...\oplus\widetilde{T}_{n-1}%
q_{p}x\text{ ; }x\in q_{p}P\oplus q_{p}\overline{P}^{c}\right\}
\]
is complemented in \textit{Lat}$(q_{p}Aq_{p})^{(n)}$. By the definition of the
transformations $\widetilde{T}_{i}$ it follows immediately that the subspace
\[
\left\{  q_{p}x\oplus T_{1}q_{p}x\oplus T_{2}q_{p}x\oplus...\oplus
T_{n-1}q_{p}x\text{ ; }x\in P\right\}
\]
is complemented in \textit{Lat}$(q_{p}Aq_{p})^{(n)}.$ If we set $p_{m}%
=r_{m}q_{m}\in\mathcal{B}$ we have that $p_{m}\nearrow I$, $p_{m}^{(n)}K$ is
complemented in \textit{Lat}$(p_{m}Ap_{m})^{(n)}$:%
\begin{align*}
p_{m}^{(n)}X  &  =(0\oplus p_{m}^{(n-1)}K_{0})+\left\{  p_{m}x\oplus
T_{1}p_{m}x\oplus...\oplus T_{n-1}p_{m}x\text{ };\text{ }x\in P\right\} \\
&  +((p_{m}P)^{c})\oplus p_{m}^{(n-1)}K_{0}^{c}).
\end{align*}
Hence:%
\[
p_{m}^{(n)}X=(p_{m}^{(n)}K)+((p_{m}P)^{c})\oplus p_{m}^{(n-1)}K_{0}^{c}).
\]
\newline and the proof of the lemma is completed.
\end{proof}

\bigskip

The following statement follows from the proof of Lemma 20.

\begin{remark}
If $A$ is as in the statement of Lemma 20 and $K\in$Lat$A^{(n)}$ for some
$n\in%
\mathbb{N}
$, then there exists an increasing sequence of projections $\left\{
p_{m}\right\}  \subset\mathcal{B},$ $p_{m}\nearrow I$ such that $p_{m}%
^{(n)}K=(0\oplus p_{m}^{(n-1)}K_{0})+\left\{  p_{m}x\oplus T_{1}p_{m}%
x\oplus...\oplus T_{n-1}p_{m}x\text{ ; }x\in P\right\}  $ where $K_{0}%
=\left\{  x\in X^{(n-1)}\text{ ; }0\oplus x\in K\right\}  $ and $T_{i}%
p_{m},1\leq i\leq n-1$, $m\in%
\mathbb{N}
,$ are bounded spectral operators of finite type on the closed $A-$invariant
subspace $p_{m}P$ that commute with $p_{m}Ap_{m}.$
\end{remark}

We can now give the proof of Theorem 15:

\bigskip

\textbf{Proof of Theorem 15}. Let $b\in$ \textit{algLat}$A$ and $K\in
$\textit{Lat}$A^{(n)}.$ We will prove by induction on $n$ that there exists an
increasing sequence of projections $\left\{  p_{m}\right\}  \subset
\mathcal{B}$ such that $p_{m}\nearrow I$ and $p_{m}^{(n)}K\in$\textit{Lat}%
$(p_{m}bp_{m})^{(n)}$ for every $m\in%
\mathbb{N}
$ and therefore $K\in$\textit{Lat}$b^{(n)}$; then apply Remark 1 to conclude
that $b\in A.$ By Remark 21, there exists an increasing sequence of
projections $\left\{  p_{m}\right\}  \subset\mathcal{B},$ $p_{m}\nearrow I$
such that $p_{m}^{(n)}K=(0\oplus p_{m}^{(n-1)}K_{0})+\left\{  p_{m}x\oplus
T_{1}p_{m}x\oplus T_{2}p_{m}x\oplus...\oplus T_{n-1}p_{m}x\text{ ; }x\in
P\right\}  $ where $K_{0}=\left\{  x\in X^{(n-1)}\text{ ; }0\oplus x\in
K\right\}  $ and $T_{i}p_{m},1\leq i\leq n-1$, $m\in%
\mathbb{N}
,$ are bounded spectral operators of finite type on the closed $A-$invariant
subspace $p_{m}P$ that commute with $p_{m}Ap_{m}.$ The induction hypothesis
and Proposition 2 (i) imply that $0\oplus p_{m}^{(n-1)}K_{0}\in$%
\textit{Lat}$(p_{m}bp_{m})^{(n)}.$ By Proposition 19 (ii) it follows that the
bounded spectral operators of finite type $T_{i}p_{m},1\leq i\leq n-1$, $m\in%
\mathbb{N}
$ commute with $p_{m}bp_{m}.$ Hence $p_{m}^{(n)}K\in$\textit{Lat}$(p_{m}%
bp_{m})^{(n)}$. Since $p_{m}\nearrow I$ and, by Remark 7 (iv), $\mathcal{B}$
is norm bounded, it follows that $K\in$\textit{Lat}$b^{(n)}$ and the result
follows. $\blacksquare$

\bigskip

\begin{corollary}
Let $A\subset B(X)$ be a strongly closed algebra that contains a complete
Boolean algebra of projections $\mathcal{B}$ of finite uniform multiplicity
with the direct sum property. If $A$ has complemented invariant subspace
lattice, then $A=A"$ where $A"$ is the bicommutant of $A.$
\end{corollary}

\begin{proof}
Follows from Theorem 15 and Remark 4.
\end{proof}

\bigskip

\bigskip%

\[
\text{{\LARGE References}}%
\]

1. W.B. Arveson, A density theorem for operator algebras, \textit{Duke Math.
J. 34 (1967), }635-647\textit{.}

2. W. G. Bade, On Boolean algebras of projections and algebras of operators,
\textit{Trans. Amer. Math. Soc. 80 (1955), }345-360\textit{.}

3. W. G. Bade, A multiplicity theory for Boolean algebras of projections in
Banach spaces, \textit{Trans. Amer. Math. Soc. 92 (1959), 508-530.}

4. J. Dieudonn\'{e}, Champs de vecteurs non localement triviaux, \textit{Arch.
Math. 75 (1956), 6-10.}

5. R. G. Douglas and C. Pearcy, Hyperinvariant subspaces and transitive
algebras, \textit{Mich. Math. J. 19 (1972), }1-12\textit{.}

6. N. Dunford and J. Schwartz, Linear operators Part III, \textit{Wiley-
Interscience,} \textit{New York- London- Sydney- Toronto, 1971.}

7. A. Feintuch and P. Rosenthal, Remarks on reductive operator algebras,
\textit{Israel J. Math. 15 (1973), }130-136\textit{.}

8. S. R. Foguel, Boolean algebras of projections of finite
multiplicity\textit{, Pacific J. Math. 9 (1959), 681-693.}

9. T. B. Hoover, Operator algebras with complemented invariant subspace
lattices, \textit{Indiana Univ. Math. J. 22 (1972/73), }1029-1035\textit{.}

10. H. Onder and M. Orhon, Transitive operator algebras on the $n-$ fold
direct sum of a cyclic Banach space, \textit{J. Operator Theory 22 (1989),
}99-107\textit{.}

11. H. Radjavi and P. Rosenthal,\ Invariant subspaces, \textit{2nd edition,}
\textit{Dover Publications, Mineola, New York, 2003.}

12. P. Rosenthal and A. R. Sourour, On operator algebras ccontaining cyclic
Boolean algebras II, \textit{J. London Math. Soc. 16 (1977), }%
501-506\textit{.}

13. M. H. Stone, Boundedness properties in function lattices, \textit{Canadian
J. Math. 1 (1949), }176-186\textit{.}

14. L. Tzafriri, Operators commuting with Boolean algebras of projections of
finite multiplicity, \textit{Pacific J. Math. }20 (1967), 571-587.

\end{document}